\documentclass{amsart}
\usepackage[T1]{fontenc}
\usepackage{latexsym,amssymb,amsmath}

\usepackage[all,ps]{xy}
\usepackage{amsthm}
\usepackage{longtable}
\usepackage{epsfig}
\usepackage{mathrsfs}
\usepackage{hhline}
\usepackage{epic}
\usepackage{enumerate}
\usepackage{graphicx}

 \newcommand{\res}{\operatorname{res}}

 \newcommand{\Alt}{\mathcal{A}}

 \newcommand{\sym}{\mathfrak{S}}

 \newcommand{\Res}{\operatorname{Res}}

 \newcommand{\Cen}{\operatorname{C}}

 \newcommand{\cyc}[1]{\langle\,#1\,\rangle}

 \newcommand{\C}{\mathbb{C}}

 \newcommand{\Z}{\mathbb{Z}}
 
 \newcommand{\Sym}{\mathfrak{S}}

 \newcommand{\Irr}{\operatorname{Irr}}
 \newcommand{\SI}{\operatorname{SI}}

 \newcommand{\la}{\lambda}

 \newcommand{\da}{\delta}
  
  \newcommand{\sa}{\sigma}

\newcommand{\tSym}{\widetilde{\Sym}}
\newcommand{\tAlt}{\widetilde{\Alt}}

\newtheorem{theorem}{Theorem}[section] 
\newtheorem{lemma}[theorem]{Lemma}     
\newtheorem{corollary}[theorem]{Corollary}
\newtheorem{proposition}[theorem]{Proposition}

\theoremstyle{definition}
\newtheorem{remark}[theorem]{Remark}

\title[]
{Basic sets for the double covering groups of the symmetric and alternating
groups in odd characteristic}
\author{Olivier Brunat}

\address{Universit\'e Paris-Diderot Paris 7\\ Institut de math\'ematiques de
         Jussieu -- Paris Rive Gauche\\ UFR de math\'e\-matiques\\ Case
7012\\ 75205 Paris Cedex 13\\
         France.}
\email{brunat@math.univ-paris-diderot.fr}

\author{Jean-Baptiste Gramain}

\address{Institute of Mathematics, 
University of Aberdeen, King's College \\
Fraser Noble Building, Aberdeen AB24 3UE, UK
}
\email{jbgramain@abdn.ac.uk}
\thanks{The second author was supported by the EPSRC grant
\emph{Combinatorial Representation Theory} EP/M019292/1}

\subjclass[2010]{Primary 20C30,\, 20C15; Secondary 20C20}

\begin{document}
\begin{abstract}
In this paper, following the methods of~\cite{BrGr}, we show that the
double covering groups of the symmetric and alternating groups have $p$-basic
sets for any odd prime $p$.
\end{abstract}
\maketitle

\section{Introduction}
\label{sec:intro}
Let $G$ be a finite group and $\mathcal C$ be a union of conjugacy
classes of $G$. For any class function $\alpha$ on $G$, define
$\res_{\mathcal C}(\alpha)=\widehat\alpha$ to be the class function on
$G$ that takes the same values as $\alpha$ on $\mathcal C$ and $0$
elsewhere.  A \emph{$\mathcal C$-basic set} of $G$ is a subset $b$ of
the set $\Irr(G)$ of complex irreducible characters of $G$ such that
$\widehat b=\{\widehat \chi\,|\,\chi\in b\}$ is a $\Z$-basis of the
$\Z$-module generated by $\widehat{\Irr}(G)=\{\widehat
\chi\,|\,\chi\in\Irr(G)\}$.

In~\cite[\S2.1]{BrGr3}, the authors develop a generalized modular theory
of $G$ with respect to $\mathcal C$. In particular, there is a partition
of $\Irr(G)$ into $\mathcal C$-blocks with respect to this $\mathcal
C$-structure. We can then analogously define the notion of $\mathcal
C$-basic set of any $\mathcal C$-block $B$.  Note that, if $b$ is a
$\mathcal C$-basic set of $G$ and $B$ is any $\mathcal C$-block of $G$,
then $b \cap B$ is a $\mathcal C$-basic set of $B$. Conversely, given a
$\mathcal C$-basic set for each $\mathcal C$-block of $G$, one recovers
a $\mathcal C$-basic set of $G$.

Let $p$ be a prime.  When $\mathcal C$ is the set of $p$-regular
elements of $G$, that is, elements whose order is prime to $p$, the
$\mathcal C$-modular theory of $G$ is the classical $p$-modular
character theory, and the notion of $\mathcal C$-basic set coincides
with that of \emph{$p$-basic set}.  One of the main challenges of
Brauer's theory is the determination of the $p$-decomposition matrix of
$G$, of which a first approximation can be obtained when one has a
$p$-basic set of $G$.

While it is unclear how to exhibit a $p$-basic set of $G$ in general,
Hiss conjectured that all finite groups do have $p$-basic sets. This
conjecture is true in the following cases: for $p$-soluble
groups~\cite[X.2.1]{Feit}, for certain finite reductive groups under
various assumptions (\cite{HissG2}, \cite{geck3D4}, \cite{GH},
\cite{GeckBS}, \cite{GeckBS3}, \cite{kltiep}, \cite{denoncin},
\cite{Br1}), and for symmetric and alternating groups
(\cite{James-Kerber}, \cite{BrGr}, \cite{BrGr2}, \cite{bessAn}).

The covering groups of the symmetric groups have a $2$-basic set
(see~\cite[Thm \,5.2]{bessolsson}). For $p$ odd, attempts have been made
in~\cite{BMO} and~\cite{ABO} to generalize the method of James and
Kerber, but these methods fail for $p\geq 7$. In this paper, we prove
the following.

\begin{theorem}
\label{thm:main}
The covering groups of the symmetric and alternating groups have
$p$-basic sets for any odd prime $p$.
\end{theorem}

In~\cite{BrGr3} (see also~\cite{KOR}), the authors gave a notion of
perfect isometries, which generalizes to $\mathcal C$ the perfect
isometries introduced by Brou\'e in~\cite{Broue}. Such a perfect
isometry $I$ furthermore satisfies $I(\widehat\chi)=\widehat{I(\chi)}$
for all $\chi\in B$, whence in particular sends a basic set to a basic
set.

To prove Theorem~\ref{thm:main}, we follow the strategy of~\cite{BrGr}.
In Section~\ref{sec:deux}, we present groups $\widetilde G$ and
$\widetilde H$ and unions of conjugacy classes $\mathcal C$ of
$\widetilde G$ and $\mathcal D$ of $\widetilde H$, and construct a
$\mathcal C$-basic set of $\widetilde G$ (see
Theorem~\ref{thm:basicset1}) and a  $\mathcal D$-basic set of
$\widetilde H$ (see Theorem~\ref{thm:basicset2}).

In Section~\ref{sec:perfiso}, we adapt the perfect isometry constructed
by Livesey in~\cite{livesey}, where he proves Brou\'e's Abelian Defect
Perfect Isometry Conjecture, and we apply the methods and results
of~\cite{BrGr3}. This produces, for any $p$-block $B$ of a double covering
group of the symmetric or alternating groups, a perfect isometry between
$B$ and a $\mathcal C$-block of $\widetilde G$ or a $\mathcal D$-block of
$\widetilde H$, from which Theorem~\ref{thm:main} is deduced. 
Finally, Section~\ref{sec:cinq} is devoted to some further remarks. In
particular, we recover some results of Olsson on the Brauer characters
in $p$-blocks of covering groups ($p$ odd), and their number. 

\section{Background}
\label{sec:background}
Let $n$ be a positive integer.
We write 
$\widetilde{\Sym}_n$ for the double covering group of the symmetric group
$\sym_n$ defined by
$$\widetilde{\Sym}_n=\left\langle z,\, t_i,\,1\leq i\leq n-1\ |\
z^2=1,\,t_i^2=z,\,(t_it_{i+1})^3=z,\,(t_it_j)^2=z\ (|i-j|\geq 2)
\right\rangle.$$
The group $\tSym_n$ and its representation theory were first studied by
I. Schur in \cite{schur}, and, unless otherwise specified, we always
refer to \cite{schur} for details or proofs.

\noindent
We recall 
that we have the following exact sequence
$$1\rightarrow \cyc{z}\rightarrow \widetilde{\Sym}_n\rightarrow
\sym_n\rightarrow 1.$$ 

We denote by $\theta:\widetilde{\Sym}_n\rightarrow\sym_n$ the natural
projection. Note that for every $s\in\sym_n$, we have
$\theta^{-1}(s)=\{\widetilde{s},z\widetilde{s}\}$, where
$\widetilde{s}\in\widetilde{\Sym}_n$ is such that
$\theta(\widetilde{s})=s$.  Whenever $G$ is a subgroup of
$\sym_n$, we set $\widetilde G=\theta^{-1}(G)$.  Any irreducible
(complex) character of $\widetilde G$ with $z$ in its kernel is simply
lifted from an irreducible character of $G$.  Any other irreducible
character $\xi$ of $\widetilde G$ is called a {\emph{spin character}},
and it satisfies $\xi(z)=-\xi(1)$.  We denote by $\SI(\widetilde G)$ the
set of irreducible spin characters of $\widetilde G$.  Define
$\varepsilon=\operatorname{sgn}\circ \, \theta$, where
$\operatorname{sgn}$ is the sign character of $\sym_n$. In particular,
if $G$ is the alternating group $\Alt_n$, then
$\widetilde{\Alt}_n=\ker(\varepsilon)$. 

Note that the restriction of $\varepsilon$ to $\widetilde G$ (again
denoted by $\varepsilon$) is a linear (irreducible) character of
$\widetilde G$, and for any spin character $\xi$ of $\widetilde G$,
$\varepsilon\otimes\xi$ is a spin character (because
$\varepsilon\otimes\chi(z)=-\varepsilon\otimes\chi(1)$).  A spin
character $\xi$ is said to be {\emph{self-associate}} if $\varepsilon
\xi=\xi$. Otherwise, $\xi$ and $\varepsilon \xi$ are called
{\emph{associate characters}}. 

If $H_1$ and $H_2$ are subgroups of $\sym_n$ acting non-trivially on
disjoint subsets of $\{1,\ldots,n\}$ and $H=H_1\times H_2\leq \sym_n$,
then Schur defined in~\cite{schur} (see also~\cite{hump}) the
\emph{twisted central product} denoted by $\hat{\times}$ and such that
$\widetilde H=\widetilde H_1\hat{\times}\widetilde H_2\leq \tSym_n$.
Schur proved that there is a surjective map $\hat\otimes:\SI(\widetilde
H_1)\times\SI(\widetilde H_2) \rightarrow \SI(\widetilde H)$. Schur
explicitly described when two pairs of spin characters have the same
image $\chi\hat{\otimes}\psi$ under this map, and how to detect whether
$\chi\hat\otimes \psi$ is self-associate or not. Furthermore, we derive
from case 1 of \S2 and the proof of Theorem~2.4 of~\cite{hump} that

\begin{proposition}
\label{prop:humph}
Let $\widetilde H_1$ and $\widetilde H_2$ be as above and not contained
in the kernel of $\varepsilon$.
Let $\chi_1$ and $\chi_2$ be irreducible spin characters of $\widetilde
H_1$ and $\widetilde H_2$. Let $\tau_1\in \widetilde H_1$ and
$\tau_2\in\widetilde H_2$ be such that
$\varepsilon(\tau_1)=\varepsilon(\tau_2)=1$. Then 
\begin{enumerate}[(1)]
\item If $\chi_1$ (or $\chi_2$) is self-associate, then
$$(\chi_1\hat\otimes\chi_2)(\tau_1,\tau_2)=\chi_1(\tau_1)\chi_2(\tau_2).$$
\item If $\chi_1$ and $\chi_2$ are non-self-associate, then 
$$(\chi_1\hat\otimes\chi_2)(\tau_1,\tau_2)=2\chi_1(\tau_1)\chi_2(\tau_2).$$
\end{enumerate}
\end{proposition}

Let $\mathcal{D}_n$ be the set of {\emph{bar partitions} of $n$,
\emph{i.e.} the partitions of $n$  with distinct parts.  Denote by
$\mathcal{D}^+_n$ (respectively $\mathcal{D}^-_n$) the subset of
$\mathcal{D}_n$ consisting of all partitions $\pi\in \mathcal{D}_n$ such
that the number of even parts of $\pi$ is even
(respectively odd).  In~\cite{schur}, Schur proved that the spin
characters of $\tSym_n$ are, up to association, labeled by
$\mathcal{D}_n$. More precisely, he showed that every
$\lambda\in\mathcal{D}^+_n$ indexes a self-associate spin character
$\xi_{\lambda}$, and every $\lambda\in\mathcal{D}^-_n$ a pair
$(\xi_{\lambda}^{+} , \xi_{\lambda}^-)$ of associate spin characters. 

The irreducible spin characters of $\tAlt_n$ are also labeled by the bar
partitions of $n$. If $\lambda\in\mathcal D_n^+$, then
$\Res_{\tAlt_n}^{\tSym_n}(\xi_{\lambda})=\zeta_{\lambda}^++\zeta_{\lambda}^-$
with $\zeta_{\lambda}^+,\,\zeta_{\lambda}^-\in\SI(\tAlt_n)$. 
If $\lambda\in\mathcal D_n^-$, then 
$\Res_{\tAlt_n}^{\tSym_n}(\xi_{\lambda}^+)=
\Res_{\tAlt_n}^{\tSym_n}(\xi_{\lambda}^-)=\zeta_{\lambda}\in\SI(\tAlt_n)$.
\smallskip

For any bar partition $\lambda=(\lambda_1>\cdots>\lambda_k>0)$ of $n$,
we set, as in the case of partitions, $|\lambda|=\sum \lambda_i$
and we define the {\emph{length}} $\ell(\la)$ of $\lambda$ by
$\ell(\lambda)=k$.  
Furthermore, we set
$\sigma(\lambda)=(-1)^{|\lambda|-\ell(\lambda)}$. With this notation, we
then have 
$\lambda\in\mathcal D_n^{\sigma(\lambda)}$ (see e.g.~\cite[p.\,45]{olsson}).

\smallskip
Let $p$ be an odd integer. 
To any $\lambda\in\mathcal D_n$, one can associate in a canonical way
its $\overline p$-core $\lambda_{(\overline p)}$ and its $\overline
p$-quotient $\lambda^{(\overline p)}$; see~\cite[p.\,28]{olsson}.
The $\overline p$-core $\lambda_{(\overline p)}$ is a bar partition of
$n-wp$, where $w$ is a non-negative integer called the $\overline
p$-weight of $\lambda$. The $\overline p$-quotient has the form
$\lambda^{(\overline p)}=(\lambda^0, \, \lambda^1, \, \ldots, \,
\lambda^{(p-1)/2})$, where $\lambda^0$ is a bar partition, the
$\lambda^i$'s are partitions for $1\leq i\leq (p-1)/2$, and the sizes of
the $\lambda_i$'s ($0\leq i\leq (p-1)/2$) add up to $w$. 
If we write, in analogy with bar partitions,
$\sigma(\lambda^{(\overline p)})=(-1)^{w-\ell(\lambda^0)}$, then we
obtain that \begin{equation}
\sigma(\lambda)=\sigma(\lambda_{(\overline{p})})
\sigma(\lambda^{(\overline{p})}).
\label{eq:sign}
\end{equation}

\medskip

\section{Some basic sets}
\label{sec:deux}
Let $w$ be a positive integer and $p$ be an odd prime.
In this section, we will focus on a group $\widetilde G$ which, when
$w<p$, is isomorphic to the normalizer of a Sylow $p$-subgroup of
$\tSym_{pw}$. We will recall the description of its irreducible
characters, and describe a union $\mathcal C$ of conjugacy classes 
and a $\mathcal C$-basic set for $\widetilde G$. 

Set $K=\Z/p\Z\rtimes \Z/(p-1)\Z$ and $G=K\wr\Sym_w$. Recall that $G$ is
a natural subgroup of $\Sym_{pw}$ and that $G=N\rtimes \Sym_w$, where
$N=K^w$.  Now, let $Q=(\Z/p\Z)^w\triangleleft N$ and
$L=\Z/(p-1)\Z\wr\Sym_w$. Then $G=Q\rtimes L$. Note that $|\widetilde
Q|=2p^w$. Thus, by Sylow's theorems, $\widetilde Q$ has a Sylow
$p$-subgroup $P$ of index $2$, whence normal and unique in $\widetilde
Q$. Since $\langle z\rangle$ is normal in $\widetilde Q$ and $\langle
z\rangle\cap P$ is trivial (as $p$ is odd),
we see that $\widetilde Q=\langle z\rangle\times P$ and that $P\simeq
Q$. Now take any $g\in\widetilde G$. Then, since $\widetilde
Q\triangleleft \widetilde G$, the group $gPg^{-1}$ is a Sylow
$p$-subgroup of $\widetilde Q$. Hence, $gPg^{-1}=P$ and $P$ is normal in
$\widetilde G$.  
Furthermore, as $Q\cap L=1$, we have $P\cap\widetilde L=1$ (because
$P\cap \widetilde L=P\cap(\widetilde Q\cap \widetilde L)=P\cap\langle
z\rangle =1$).
Hence
\begin{equation}
\label{eq:semi}
\widetilde G=\widetilde N\,\widetilde \Sym_w=P\rtimes \widetilde L. 
\end{equation}
The irreducible spin characters of $\widetilde G$ were constructed by
Michler and Olsson in~\cite{Michler-Olsson} using the Clifford's theory
(see for example~\cite[Chap\,6]{isaacs}) of $\widetilde G$ with respect to the
normal subgroup $\widetilde N$. The parametrization they obtain is as
follows.

\begin{proposition}
\label{prop:chargtilde} 
The irreducible spin characters of $\widetilde G$ can be labeled
explicitly by the $\overline p$-quotients of $w$. 
If $q=(q^0,q^1,\ldots,q^{(p-1)/2})$ is a $\overline p$-quotient of $w$,
then $q$ labels two associate characters $\psi_q^+$ and $\psi_q^-$ if
and only if $\sigma(q)=-1$; otherwise, $q$ labels a unique
self-associate character $\psi_q$.
\end{proposition}

\noindent For the second statement, an explanation can be found for example
in~\cite[p.\,422]{jbnavarro}.

\begin{proposition}
\label{prop:kernel}
The spin characters of $\widetilde G$ which have $P$ in their kernel are
exactly those labeled by $\overline p$-quotients of the form
$q=(\emptyset,q^{1},\ldots,\,q^{(p-1)/2})$. 
\end{proposition}

\begin{proof}
Take any $\psi\in\Irr(\widetilde G)$.
Since $P\leq \widetilde N$, to study the restriction of $\psi$ to $P$,
we will look at the restriction of $\psi$ to $\widetilde N$. On the
other hand, the restriction to $\widetilde N$ is given by Clifford's
theorem~\cite[Thm\, 6.2]{isaacs} which goes as follows. There exist a
unique $\widetilde G$-orbit~$\Omega_{\psi}$ of irreducible characters
of $\widetilde N$ and a unique positive integer $e_{\psi}$ such that
\begin{equation}
\label{eq:clifford}
\Res_{\widetilde N}^{\widetilde
G}(\psi)=e_{\psi}\,\sum_{\chi\in\Omega_{\psi}}\,\chi. 
\end{equation}
It follows from~(\ref{eq:clifford}), the facts that $\psi$ is a spin
character and that all $\chi\in\Omega_{\psi}$ take the same value on $z$
that $\Omega_{\psi}$ consists only of spin characters.
We will thus now recall the description given in~\cite{Michler-Olsson}
of the $\widetilde G$-orbits of spin characters
of $\widetilde N$.

The irreducible spin characters of $\widetilde N$ are of the form
$\chi_1\hat\otimes \cdots\hat\otimes \chi_w$, where
$\chi_i\in\SI(\widetilde K)$ for $1\leq i\leq w$. As above, we have $\widetilde K=\langle \tau\rangle\rtimes
\theta^{-1}(\Z/(p-1)\Z)$, where $\tau$ is an element or order $p$. 
In \cite{Michler-Olsson}, Michler and Olsson show that $\widetilde K$
has exactly $p$ spin characters: a unique self-associate $\eta_0$, of
degree $p-1$ and with $\eta_0(\tau)=-1$, 
and $(p-1)/2$ pairs of associate linear characters
$\eta_1^{+},\,\eta_1^{-},\ldots,\eta_{(p-1)/2}^{+},\,\eta_{(p-1)/2}^-$, 
which are exactly those having $\langle \tau\rangle$ in their kernel.  

For any decomposition $t=(t_0,\,t_1,\,\ldots,t_{(p-1)/2})$ of $w$, we
set
\begin{equation}
\label{eq:orbite}
\theta_t=\underbrace{\eta_0\hat\otimes\cdots\hat\otimes\eta_0}_{t_0\
\text{factors}}\hat\otimes
\underbrace{\eta_1^+\hat\otimes\cdots\hat\otimes\eta_1^+}_{t_1\
\text{factors}}\hat\otimes\cdots\hat\otimes
\underbrace{\eta_{(p-1)/2}^+\hat\otimes\cdots\hat\otimes\eta_{(p-1)/2}^+}_{t_{(p-1)/2}\
\text{factors}}\in\SI(\widetilde N).
\end{equation}

Now suppose $\psi$ is labeled by the $\overline p$-quotient
$q=(q^0,\ldots,q^{(p-1)/2})$ of $w$.
From~\cite[Prop\,3.12]{Michler-Olsson} and the construction of $\psi$,
we get that, if $|q^{0}|>1$, then
$\theta_{(|q^0|,\,\ldots,\,|q^{(p-1)/2}|)}$ is a representative for
$\Omega_{\psi}=\Omega_{\varepsilon\psi}$ (even if
$\psi\neq\varepsilon\psi$). If $|q^0|\leq 1$ and $w-|q^0|$
is odd, then $\psi\neq\varepsilon\psi$ and representatives for
$\Omega_{\psi}$ and $\Omega_{\varepsilon\psi}$ are given by
$\theta_{(|q^0|,\,\ldots,\,|q^{(p-1)/2}|)}$ and
$\varepsilon\theta_{(|q^0|,\,\ldots,\,|q^{(p-1)/2}|)}$ (not necessarily
in this order). In all cases, we write $\theta_{\psi}$ for the
representative of $\Omega_{\psi}$.

Note that $Q=\langle u_1\rangle\times\cdots\times\langle u_w\rangle$,
where $u_i$ has order $p$ for $1\leq i\leq w$. For $1\leq i\leq w$, let
$\tau_i$ be the unique element of order $p$ in $\theta^{-1}(\{u_i\})$.
Then $$P=\langle \tau_1\rangle\times\cdots\times\langle \tau_w\rangle.$$
In particular, $P\leq\tAlt_{pw}$.

We now have
$\psi(1)=e_{\psi}\,\sum_{\chi\in\Omega_{\psi}}\chi(1)=
e_{\psi}\theta_{\psi}(1)|\Omega_{\psi}|$.
Furthermore, iterating
Proposition~\ref{prop:humph} shows that, for any $\chi\in\Omega_{\psi}$,
there is a non-negative integer $\alpha_{\chi}$ such that
$\chi(1)=2^{\alpha_{\chi}}(p-1)^{|q^0|}$ and
$\chi(\tau_1,\ldots,\tau_w)=2^{\alpha_{\chi}}(-1)^{|q^0|}$. 
Since $\chi(1)$ does not depend on the choice of $\chi$ in
$\Omega_{\psi}$, neither does $\alpha_{\chi}$ which we thus denote by
$\alpha$. In particular, we have
$$\psi(1)=e_{\psi}2^{\alpha}(p-1)^{|q^0|}|\Omega_{\psi}|\quad\text{and}\quad
\psi(\tau_1\cdots\tau_w)=e_{\psi}2^{\alpha}(-1)^{|q^0|}|\Omega_{\psi}|.
$$
It follows that, if $q^0\neq\emptyset$, then
$\psi(\tau_1\cdots\tau_w)\neq\psi(1)$, whence $P$ is not contained in
$\ker(\psi)$. If $q^{0}=\emptyset$, then a similar computation shows
that, for any $(x_1,\ldots,x_w)\in P$, we have $\psi(x_1\cdots
x_w)=e_{\psi}2^{\alpha}|\Omega_{\psi}|=\psi(1)$, as required.
\end{proof}

We define $\mathcal C$ to be the union of conjugacy classes of
$\widetilde G$ which have a representative in $\widetilde L$.

\begin{theorem}
\label{thm:basicset1}
The set of irreducible spin characters of $\widetilde G$ is a union of
$\mathcal C$-blocks, with $\mathcal C$-basic set the set $B$ of spin
characters labeled by $\overline p$-quotients of the form
$q=(\emptyset,q^1,\ldots,q^{(p-1)/2})$. 
\end{theorem}

\begin{proof}
Let $\chi$ be a non-spin character of $\widetilde G$ and $\psi$ be a
spin character of $\widetilde G$. For any $g\in\widetilde G$, one
has $\chi(g)=\chi(zg)$ and $\psi(g)=-\psi(zg)$. Also, since
$z\in\widetilde L$, we have that $\mathcal C=z\mathcal C$. Thus,
writing, for any $A\subseteq \widetilde G$, $\langle
\psi,\,\chi\rangle_A=\frac 1{|\widetilde G|}\sum_{x\in
A}\psi(x)\overline{\chi(x)}$, we have
$$2|\widetilde{G}|\langle \psi,\chi\rangle_{\mathcal C}=
|\widetilde{G}|(\langle \psi,\chi\rangle_{\mathcal C} +\langle
\psi,\chi\rangle_{z\mathcal C})=\sum_{g\in\mathcal
C}(\psi(g)+\psi(zg))\overline{\chi(g)}=0.$$
This proves that the set of spin characters of $\widetilde G$ is a union
of $\mathcal C$-blocks (see~\cite[Prop 2.14]{BrGr3} and
\cite[\S1]{KOR}). 
By Lemma~4.1 and Proposition~4.2 of~\cite{BrGr}, 
the set of irreducible
characters of $\widetilde G$ with $P$ in their kernel is a $\mathcal
C$-basic set of $\widetilde G$. The intersection of this basic set with
the set of spin characters of $\widetilde G$, which is given by
Proposition~\ref{prop:kernel}, is thus a $\mathcal C$-basic set for the
set of spin characters of $\widetilde G$. 
\end{proof}

Let $H=G\cap \Alt_{pw}$. By Clifford's theory between $\widetilde G$ and
$\widetilde H$, the spin irreducible characters of $\widetilde H$ are
again labeled by the $\overline p$-quotients of $w$. Any $\overline
p$-quotient $q$ of $w$ labels two characters $\varphi_q^+$ and
$\varphi_q^-$ of $\widetilde H$ when $\sigma(q)=1$, and one character
$\varphi_q$ otherwise. Set $\mathcal D=\mathcal C\cap \widetilde H$.  

\begin{theorem}
\label{thm:basicset2}
The set of irreducible spin characters of $\widetilde H$ is a union of
$\mathcal D$-blocks, with $\mathcal D$-basic set the set $B'$ of spin
characters labeled by $\overline p$-quotients of the form
$q=(\emptyset,q^1,\ldots,q^{(p-1)/2})$. 
\end{theorem}

\begin{proof}
Since $z$ lies in $\mathcal D$, we conclude as in
Theorem~\ref{thm:basicset1} that 
the set of irreducible spin characters of $\widetilde H$ is a union of
$\mathcal D$-blocks. Furthermore, a simple counting argument shows that
$$\widetilde H=P\rtimes (\widetilde L\cap \tAlt_{pw}).$$

Let $q=(q^0,q^1,\ldots,q^{(p-1)/2})$ be a $\overline p$-quotient of $w$.
If $\sigma(q)=-1$, then $\varphi_q=\Res_{\widetilde H}^{\widetilde
G}(\psi_q^+)$ has $P$ in its kernel if and only if $\psi_q^+$ does.  If
$\sigma(q)=1$, then $\Res_{\widetilde H}^{\widetilde
G}(\psi_q)=\varphi_q^++\varphi_q^-$, with
$\varphi_q^+(1)=\varphi_q^-(1)=\psi_q(1)/2$. Let $x\in P$. 
We have $\psi_q(x)=\psi_q(1)$ if and only if
$\varphi_q^+(x)+\varphi_q^-(x)=\psi_q(1)$; 
since, by~\cite[Lemma\,2.15]{isaacs}, $|\varphi_q^+(x)|\leq \psi_q(1)/2$ and 
$|\varphi_q^-(x)|\leq \psi_q(1)/2$, this holds if and only if
$\varphi_q^+(x)=\varphi_q^-(x)=\psi_q(1)/2$. 

This proves that any spin character of $\widetilde H$ labeled by $q$ has
$P$ in its kernel if and only if $q^0=\emptyset$, and we can conclude as
in the proof of Theorem~\ref{thm:basicset1}.
\end{proof}

\begin{remark}
\label{rk:brgrrevisited}
The same arguments as in the proofs of Theorems~\ref{thm:basicset1}
and~\ref{thm:basicset2} show that the set of characters of $G$ with $Q$
in their kernel is a $\theta(\mathcal C)$-basic set of $G$ (which is the
same as that found in~\cite{BrGr}), and  that the set of characters of
$H$ with $Q$ in their kernel is a $\theta(\mathcal D)$-basic set of $H$.
\end{remark}

\section{Perfect isometries and basic sets for the covering groups}
\label{sec:perfiso}

We start by recalling the definitions of perfect isometry and of
Brou\'e isometry, and we prove some of their properties which
will be useful later.
 
Let $G$ and $G'$ be two finite groups, and $\mathcal C$ and $\mathcal
C'$ unions of conjugacy classes of $G$ and $G'$ respectively. Let $B$
(resp. $B'$) be a union of $\mathcal C$-blocks of $G$ (resp. of
$\mathcal C'$-blocks of $G'$); see~\cite[p.\,4]{BrGr3}.  
An isometry $I:\C B\rightarrow\C B'$ such that
$I(\Z B)=\Z B'$ is
a perfect isometry if  
\begin{equation}
\label{eq:isoperf}
I\circ\res_{\mathcal C}=\res_{\mathcal C'}\circ I.
\end{equation}

\begin{remark}
\label{rk:transitivite}
If $I:\C B\rightarrow \C B'$ and $J:\C B'\rightarrow \C B''$ are 
perfect isometries, then it follows easily from~(\ref{eq:isoperf}) that
$J\circ I:\C B\rightarrow \C B''$ is a perfect isometry.  
\end{remark}

To $I$, we associate $\widehat I\in\C(B\times B')$ as
in~\cite[\S2.3]{BrGr3} which satisfies 
\begin{equation}
\label{eq:formI}
\widehat I=\sum_{\chi\in B}\overline{\chi}\otimes I(\chi).
\end{equation} 
Recall that (see~\cite[Eq\,(16)]{BrGr3}), for any $\psi\in\C\Irr(B)$ and 
$y\in G'$, we have
\begin{equation}
I(\psi)(y)=\frac{1}{|G|}\sum_{x\in G}\widehat{I}(x,y)\psi(x).
\label{eq:indgene}
\end{equation}

Using~(\ref{eq:formI}) and~(\ref{eq:indgene}), an easy computation shows
that, if $B''$ is a union of $\mathcal C''$-blocks of $G''$ and $J:\C
B'\rightarrow \C B''$ is a perfect isometry, then 
for any $(x,z)\in G\times G''$
\begin{equation}
\label{eq:chapeaucomp}
\widehat{J\circ I}(x,z)=\frac 1{|G'|}\sum_{y\in G'}\widehat
I(x,y)\widehat J (y,z). 
\end{equation}

If $p$ is a prime and $(K,\mathcal R,k)$ is a splitting
$p$-modular system for $G$ and $G'$, and if $\mathcal C$ and $\mathcal
C'$ are the sets of $p$-regular elements of $G$ and $G'$, then 
$I$ is a Brou\'e isometry if the following two properties are satisfied.
\begin{enumerate}[(i)]
\item For every $(x,x')\in G\times G'$, $\widehat{I}(x,x')$ lies in
$|\Cen_{G}(x)|\mathcal R\cap |\Cen_{G'}(x')|\mathcal R$.
\item If $\widehat I(x,x')\neq 0$, then $x$ and $x'$ are either both
$p$-regular or both $p$-singular.
\end{enumerate}
Note that, in this case, $I$ is also a perfect isometry in the above
sense since (ii) implies~(\ref{eq:isoperf}).

\begin{lemma}
\label{lem:transitivite}
If $I:\C B\rightarrow \C B'$ and $J:\C B'\rightarrow \C B''$ are Brou\'e
isometries, then $J\circ I:\C B\rightarrow \C B''$ is a Brou\'e
isometry.  
\end{lemma}

\begin{proof}
Take any $x\in G$ and $z\in G''$. First note that $y\mapsto \widehat
I(x,y)\widehat J (y,z)$ is a class function of $G'$. Thus,
by~(\ref{eq:chapeaucomp}), 
$$\frac{\widehat{J\circ I}(x,z)}{|\Cen_{G}(x)|}=\sum_{y\in G'/\sim}
\frac{\widehat I(x,y)}{|\Cen_{G}(x)|}\frac{\widehat J(y,z)}{|\Cen_{G'}(y)|}\in\mathcal R.$$
A similar argument shows that $\widehat{J\circ I}(x,z)\in
|\Cen_{G}(z)|\,\mathcal R$, whence $\widehat{J\circ I}$ satisfies (i). 
Now suppose $\widehat{J\circ I}(x,z)\neq 0$. Then,
by~(\ref{eq:chapeaucomp}), there exists $y\in G'$ such that
$\widehat{I}(x,y)\neq 0$ and $\widehat{J}(y,z)\neq 0$; (ii) follows.
\end{proof}

\begin{remark}
\label{rk:invid}
Given a Brou\'e isometry $I$, a formula for $\widehat{I^{-1}}$ can be
found in~\cite[Remark~\,2.12]{BrGr3}. It easily implies that $I^{-1}$ is
a Brou\'e isometry. 
\end{remark}

\begin{theorem}
\label{thm:transitivite}
Let $p$ be an odd prime and $B$ be a $p$-block of $\tSym_n$ such that
$\{\xi_{\lambda}^+,\,\xi_{\lambda}^-\}\subseteq B$ for some
$\lambda\in\mathcal D_n^-$. 
Let $J:\C B\rightarrow\C B$ be the isometry given by
$J(\xi_{\lambda}^+)=\xi_{\lambda}^-$,
$J(\xi_{\lambda}^-)=\xi_{\lambda}^+$, and $J(\xi)=\xi$ if 
$\xi\notin\{\xi_{\lambda}^+,\,\xi_{\lambda}^-\}$.
Then $J$ is a Brou\'e isometry.
\end{theorem}

\begin{proof}
Recall that a set of conjugacy classes of $\tSym_n$ (the split classes)
outside which all spin characters vanish can be labeled by the elements
of $\mathcal O_n\cup\mathcal D_n^-$, where $\mathcal O_n$ is the set of
partitions of $n$ in odd parts. Each $\pi\in \mathcal O_n\cup\mathcal
D_n^-$ labels two conjugacy classes with representatives $t_{\pi}$ and
$zt_{\pi}$ such that the following hold.
For all $\pi\in\mathcal O_n$,
$\xi_{\lambda}^+(t_{\pi})=\xi_{\lambda}^-(t_{\pi})$ and
$\xi_{\lambda}^+(zt_{\pi})=\xi_{\lambda}^-(zt_{\pi})$. 
For all $\pi\in\mathcal D_n^-$,
$$\xi_{\lambda}^+(t_{\pi})=\xi_{\lambda}^-(zt_{\pi})=
-\xi_{\lambda}^+(zt_{\pi})=-\xi_{\lambda}^-(t_{\pi})=
\delta_{\pi\lambda}i^{(n-k+1)/2}\sqrt{z_\lambda/2},$$
where $\lambda=(\lambda_1>\cdots>\lambda_k>0)$ and
$z_{\lambda}=\lambda_1\cdots\lambda_k$. 

Denote by $I$ the identity on $\C B$. By~\cite[Theorem\, 4.21]{BrGr3},
$I$ is a Brou\'e isometry. First suppose $\pi=\pi'=\lambda$
and take $\nu\in\{0,\,1\}$. Then, by~(\ref{eq:formI}), 
\begin{align*}
\widehat J(t_{\lambda},z^{\nu}t_{\lambda})-
\widehat I(t_{\lambda},z^{\nu}t_{\lambda})
=&\left(\overline{\xi_{\lambda}^+(t_{\lambda})}\xi_{\lambda}^-(z^{\nu}t_{\lambda})+
\overline{\xi_{\lambda}^-(t_{\lambda})}\xi_{\lambda}^+(z^{\nu}t_{\lambda})\right)\\
&-\left({\xi_{\lambda}^+(t_{\lambda})}\xi_{\lambda}^+(z^{\nu}t_{\lambda})+
{\xi_{\lambda}^-(t_{\lambda})}\xi_{\lambda}^-(z^{\nu}t_{\lambda})\right).
\end{align*}

If $(n-k+1)/2$ is odd, then
$\overline{\xi_{\lambda}^+(t_{\lambda})}={\xi_{\lambda}^-(t_{\lambda})}$
and
$\overline{\xi_{\lambda}^-(t_{\lambda})}={\xi_{\lambda}^+(t_{\lambda})}$,
so that $
\widehat J(t_{\lambda},z^{\nu}t_{\lambda})-
\widehat I(t_{\lambda},z^{\nu}t_{\lambda})
=0$. Thus (i) and (ii) hold for $\widehat J$
(and $x=t_{\lambda}$, $x'=z^{\nu}t_{\lambda}$).\\

If $(n-k+1)/2$ is even, then
$\overline{\xi_{\lambda}^+(t_{\lambda})}={\xi_{\lambda}^+(t_{\lambda})}$
and
$\overline{\xi_{\lambda}^-(t_{\lambda})}={\xi_{\lambda}^-(t_{\lambda})}$,
so that 
\begin{align*}
\widehat J(t_{\lambda},z^{\nu}t_{\lambda})-
\widehat I(t_{\lambda},z^{\nu}t_{\lambda})
&=
\left( \xi_{\lambda}^+(t_{\lambda})-\xi_{\lambda}^-(t_
{\lambda})\right)\left( \xi_{\lambda}^-(z^{\nu}
t_{\lambda})-\xi_{\lambda}^+(z^{\nu}t_ {\lambda})  \right)\\
&=2\xi_{\lambda}^+(t_{\lambda})2\xi_{\lambda}^-(z^{\nu}t_{\lambda})\\
&=(-1)^{\nu +1}4\xi_{\lambda}^+(t_{\lambda})^2\\
&=\pm 2 z_{\lambda}.
\end{align*}
Since $2z_{\lambda}=|\Cen_{\tSym_n}(t_{\lambda})|$, this shows that (i)
holds for $\widehat J$
(and $x=t_{\lambda}$, $x'=z^{\nu}t_{\lambda}$).
Note that (ii) is automatic in this case since $t_{\lambda}$ and
$z^{\nu}t_{\lambda}$ are both $p$-regular or both $p$-singular.

Suppose now that $\pi,\,\pi'\in\mathcal O_n\cup\mathcal
D_n^-\backslash\{\lambda\}$ and $\nu,\,\nu'\in\{0,\,1\}$.
Then $$\widehat J(z^{\nu}t_\pi,\,z^{\nu'}t_{\pi'})=\widehat
I(z^{\nu}t_\pi,\,z^{\nu'}t_{\pi'}).$$
In particular, (i) and (ii) are true for $\widehat J$ in this case.

Finally, take $x\in\tSym_n$ such that exactly one of $x$ and
$t_{\lambda}$ is $p$-regular. Then $\widehat I(x,t_{\lambda})=0$. On the
other hand, 
\small
\begin{align*}
\widehat J(x,\,t_{\lambda})-
\widehat I(x,\,t_{\lambda})
&=
\left(\overline{\xi_{\lambda}^+(x)}\xi_{\lambda}^-(t_{\lambda})+
\overline{\xi_{\lambda}^-(x)}\xi_{\lambda}^+(t_{\lambda})\right)
-\left(\overline{\xi_{\lambda}^+(x)}\xi_{\lambda}^+(t_{\lambda})+
\overline{\xi_{\lambda}^-(x)}\xi_{\lambda}^-(t_{\lambda})\right)\\
&=
\left(\overline{\xi_{\lambda}^+(x)}-\overline{\xi_{\lambda}^-(x)}\right)
\left( \xi_{\lambda}^-(
t_{\lambda})-\xi_{\lambda}^+(t_ {\lambda})  \right),
\end{align*}
\normalsize
and $\overline{\xi_{\lambda}^+(x)}-\overline{\xi_{\lambda}^-(x)}=0$
since $x$ is conjugate neither to $t_\lambda$ nor to $zt_\lambda$.
Hence $\widehat J(x,t_\lambda)=0$, and (i) and (ii) hold in this case.
More generally, whenever $x\notin \{t_{\lambda},\,zt_\lambda\}$ and
$\nu\in\{0,\,1\}$,
$$\widehat J(x,z^{\nu}t_\lambda)=\widehat I(x,z^{\nu}t_\lambda),$$
whence (i) and (ii) hold. The result follows.
\end{proof}
 
From now on, we fix an integer $n\geq 2$ and an odd prime $p$.
If $B$ is any $p$-block of $\tSym_n$, then $B$ contains either no or
only spin characters. In the former case, $B$ coincides with a $p$-block
of $\Sym_n$; in the latter, we say that $B$ is a {\emph{spin block}}.
Similarly, any $p$-block of $\tAlt_n$ contains either no spin
character, and coincides with a $p$-block of $\Alt_n$, or only spin
characters, and is then called a spin block.
The spin blocks of $\tSym_n$ and $\tAlt_n$ are described by the
following:

\begin{theorem}[Morris~\cite{Morris},\,Humphreys~\cite{Humphreys-Blocks}]\label{thm:maurice}
Let $\chi$ and $\psi$ be two spin characters of $\tSym_n$, or two spin
characters of $\tAlt_n$, labeled by bar partitions $\la$ and $\mu$
respectively. Then $\chi$ is of $p$-defect
0 (and thus alone in its $p$-block) if and only if $\la$ is a $\overline
p$-core.
If $\la$ is not a $\overline p$-core, then $\chi$ and $\psi$ belong to the same
$p$-block if and only if $\la_{(\overline p)}=\mu_{(\overline p)}$.
\end{theorem}
One can therefore define the {\emph{$\overline p$-core}} of a spin block $B$ and
its {\emph{$\overline p$-weight}}, as well as its {\emph{sign}}
$\sa(B)=\sa( \la_{(\overline p)})$ (for any bar partition $\la$ labeling some
character $\chi \in B$).


\begin{theorem}
\label{thm:isoparfaite}
Let $B$ be any spin $p$-block of $\tSym_n$ of $\overline p$-weight
$w>0$. Let $\widetilde G$, $\widetilde H$, $\mathcal C$ and $\mathcal D$
be as in~\S\ref{sec:deux}, and let $\da_{\overline p}(\la)$ be the
{\emph{relative sign}} of a bar partition $\lambda$ 
introduced  by Morris and Olsson in
\cite{morrisolsson}. 

\begin{enumerate}[(i)]
 \item If $\sa(B)=1$, then $I:\C B\rightarrow\C\SI(\widetilde G)$
given by 
\begin{align*}
I(\xi_{\lambda})&=\delta_{\overline
p}(\lambda)(-1)^{|\lambda^0|}\psi_{\lambda^{(\overline p)}}\quad\text{if
}\sigma(\lambda)=1,\\
 I(\xi_{\lambda}^{\pm})&=\delta_{\overline
p}(\lambda)(-1)^{|\lambda^0|}\psi_{\lambda^{(\overline p)}}^{\pm}\quad\text{if
}\sigma(\lambda)=-1,
\end{align*}
is a perfect isometry with respect to $p$-regular elements of $\tSym_n$
and $\mathcal C$.
 \item If $\sa(B)=-1$, then $I:\C B\rightarrow\C\SI(\widetilde H)$
given by 
\begin{align*}
I(\xi_{\lambda})&=\delta_{\overline
p}(\lambda)(-1)^{|\lambda^0|}\varphi_{\lambda^{(\overline p)}}\quad\text{if
}\sigma(\lambda)=1,\\
 I(\xi_{\lambda}^{\pm})&=\delta_{\overline
p}(\lambda)(-1)^{|\lambda^0|}\varphi_{\lambda^{(\overline p)}}^{\pm}\quad\text{if
}\sigma(\lambda)=-1,
\end{align*}
is a perfect isometry with respect to $p$-regular elements of $\tSym_n$
and $\mathcal D$.
\end{enumerate}

\begin{proof}
(i) By~\cite[Thm\,10.1]{livesey}, the isometry $I_0:\C B_0\rightarrow\C\SI(\widetilde G)$
given by 
\begin{align*}
I_0(\xi_{\mu})&=\delta_{\overline
p}(\mu)(-1)^{w(p^2-1)/8+w+|\mu^0|}\psi_{\mu^{(\overline p)}}\quad\text{if
}\sigma(\mu)=1,\\
I_0(\xi_{\mu}^{\pm})&=\delta_{\overline
p}(\mu)(-1)^{w(p^2-1)/8+w+|\mu^0|}\psi_{\mu^{(\overline
p)}}^{\pm\delta_{\overline p}(\mu)(-1)^w}\quad\text{if
}\sigma(\mu)=-1,
\end{align*}
is a Brou\'e isometry when $w<p$ and $B_0$ is the principal spin
$p$-block of $\tSym_{pw}$. Note that, by definition of
$\sigma(\mu^0)$, the sign $\eta_{\overline
p}(\mu)$ in~\cite[Thm\,10.1]{livesey} does not in fact depend on
$\sigma(\mu^0)$.
Now, the proof of~\cite[Thm\,10.1]{livesey} carries forward to the case
where $w\geq p$ to show that $J$ is a perfect isometry  with respect to $p$-regular elements of $\tSym_n$
and $\mathcal C$. Using the fact that, if $J$ is perfect isometry, then
so is $-J$, we deduce that $I_1=(-1)^{w(p^2-1)/8+w}I_0$ is again a perfect
isometry with respect to $p$-regular elements of $\tSym_n$
and $\mathcal C$. 
By~\cite[Thm\,4.21]{BrGr3}, there exists a perfect isometry $I_2$ with
respect to $p$-regular elements between $B_0$ and $B$ ($I_2$ associates
characters of $B_0$ and $B$ labeled by bar partitions with the same
$\overline p$-quotient).
Furthermore, by composing iteratively perfect isometries of the type
described in Theorem~\ref{thm:transitivite}, by
Remark~\ref{rk:transitivite}, we construct a perfect isometry $I_3:\C
B\rightarrow \C B$ with
respect to $p$-regular elements and such that $I_3\circ I_2\circ
I_1^{-1}:\SI(\widetilde G)\rightarrow \C B$ given by
$\psi_{\lambda^{(\overline p)}}\mapsto \xi_{\lambda}$ if
$\sigma(\lambda)=1$  
and $\psi_{\lambda^{(\overline p)}}^{\pm}\mapsto \xi_{\lambda}^{\pm}$ if
$\sigma(\lambda)=-1$ is a perfect isometry (by composition and
inversion) with respect to $\mathcal C$ and $p$-regular elements of
$\tSym_n$.  
The result follows from the fact that $I=(I_3\circ I_2\circ
I_1^{-1})^{-1}$.

(ii) The proof is similar. We start from the Brou\'e isometry between
$\SI(\widetilde H)$ and the principal spin $p$-block $B_0^*$ of
$\tAlt_{pw}$ given in~\cite[Thm\,10.2]{livesey}, whose proof can be
generalized to the case $w\geq p$ to produce a perfect isometry with
respect to $\mathcal D$ and the set of $p$-regular elements of
$\tAlt_{pw}$. The analogue of the perfect isometry $I_2$ of case (i) is
the perfect isometry $B_0^*\rightarrow B$ given
in~\cite[Thm\,4.21]{BrGr3}. 
\end{proof}
\end{theorem}

\begin{remark}
\label{rk:broue}
If $w<p$, then the same proof, but using Lemma~\ref{lem:transitivite},
Remark~\ref{rk:invid} and Theorem~\ref{thm:transitivite}, shows that the
isometries given in Theorem~\ref{thm:isoparfaite} are Brou\'e
isometries.
\end{remark}

\begin{corollary}
\label{cor:basic}
If $B$ is any spin $p$-block of $\tSym_n$ or of $\tAlt_n$ of $\overline
p$-weight $w$, then the set of characters of $B$ labeled by bar
partitions $\lambda$ such that $\lambda^{0}=\emptyset$ is a $p$-basic
set of $B$.
\end{corollary}

\begin{proof}
In the case of $\tSym_n$, this follows from
Theorem~\ref{thm:isoparfaite}, Theorems~\ref{thm:basicset1}
and~\ref{thm:basicset2}, and~\cite[Prop\,2.2]{BrGr} when $w>0$. If
$w=0$, then $B$ contains a single character and the result is obvious.
If $B$ is spin $p$-block of $\tAlt_n$, then \cite[Thm\,4.21]{BrGr3}
provides a perfect isometry from $B$ to a spin $p$-block of same
$\overline p$-weight $w$ of some $\tSym_m$, and which preserves the
labeling of characters.  The result then follows from the first part of
the proof and~\cite[Prop 2.2]{BrGr}. 
\end{proof}

Any spin $p$-block of $\tSym_n$ or of $\tAlt_n$ has a $p$-basic set by
Corollary~\ref{cor:basic}, and any other block has a $p$-basic set
by~\cite{BrGr}.  Hence, $\tSym_n$ and $\tAlt_n$ have $p$-basic sets,
which proves Theorem~\ref{thm:main}.

\section{Some further remarks}
\label{sec:cinq}

First notice that, by construction, the $p$-basic set $b_S$ of $\tSym_n$ we
produce is stable under tensor product by $\varepsilon$. Furthermore,
when restricted to $\tAlt_n$, it gives our $p$-basic set $b_A$ of
$\tAlt_n$, that is, the elements of $b_A$ are exactly the irreducible
constituents of the restrictions to $\tAlt_n$ of the elements of $b_S$.
\smallskip

By Remark~\ref{rk:brgrrevisited}, and using the perfect isometry
(constructed in~\cite[Thm\,5.12]{BrGr3}) between a $p$-block of $\Alt_n$
and $\Irr(H)$ with respect to $p$-regular elements of $\Alt_n$ and
$\theta(\mathcal D)$, one can directly obtain the $p$-basic set of
$\Alt_n$ constructed in~\cite[Thm 5.2]{BrGr}. The advantage of this new
method is that we need not worry about Property (2)
in~\cite[Thm\,1.1]{BrGr} which is automatically satisfied.\smallskip

If $B$ is any spin $p$-block of
$\tSym_n$ or $\tAlt_n$, then the characters in the $p$-basic set $b$ of $B$
obtained in Corollary~\ref{cor:basic} are labeled by bar partitions
$\lambda$ such that $\lambda^{(\overline
p)}=(\lambda^0=\emptyset,\lambda^1,\ldots,\lambda^{(p-1)/2})$, or,
equivalently, by the $(p-1)/2$-quotients of $w$. Furthermore, for each
such $\lambda$, one has $$\sigma(\lambda^{(\overline
p)})=(-1)^{w-\ell(\lambda^0)}=(-1)^w,$$ 
so that all quotients label two spin characters of $\tSym_n$ and one of
$\tAlt_n$ is $w$ is odd and $\sigma(B)=1$, or if $w$ is even and
$\sigma(B)=-1$, and one spin character of $\tSym_n$ and two of
$\tAlt_n$ is $w$ is odd and $\sigma(B)=-1$, or if $w$ is even and
$\sigma(B)=1$.
In particular, all the characters in $b$ are self-associate, or none is.
Since $b$ is $\varepsilon$-stable, the proof of~\cite[Prop 6.1]{BrGr}
implies that the same is true for the set of irreducible Brauer
characters in $B$, that is, all are self-associate or none is. As a
consequence, we also obtain a formula for the number of irreducible
Brauer characters in $B$. We therefore recover results of Olsson
(see~\cite[Prop (3.21) and (3.22)]{olssonNumber}) obtained by
sophisticated methods.
\medskip

We end this paper with some considerations about what we like to call
\emph{Olsson's Principle}. Roughly speaking, Olsson's Principle predicts
that, for any statement which holds for $\Sym_n$, an analogous spin statement
must hold for $\tSym_n$, even though one might have to be careful with
signs along the proof.  

We now present a few illustrations of this principle which appear
in our work.
Recall that the irreducible characters of $\sym_n$ are parametrized by
the set of partitions of $n$. For any prime $p$, a partition $\lambda$
of $n$ is uniquely determined by its $p$-core $\lambda_{(p)}$ and its
$p$-quotient $\lambda^{(p)}=(\lambda^1,\ldots,\lambda^p)$ which is a
$p$-tuple of partitions whose sizes add up to $w$, the $p$-weight of
$\lambda$. 
The celebrated \emph{Nakayama Conjecture} states that two irreducible
characters $\chi_{\lambda}$ and $\chi_{\mu}$ of $\Sym_n$ belong to the
same $p$-block if and only if $\lambda_{(p)}=\mu_{(p)}$. The spin
analogue is given by the \emph{Morris Conjecture} (Theorem~\ref{thm:maurice}).
Now, the $p$-basic set of $\Sym_n$ constructed in~\cite{BrGr} consists
of those characters which are labeled by partitions $\lambda$ with
$\lambda^{(p+1)/2}=\emptyset$, \emph{i.e.} without symmetric diagonal
$(p)$-hooks.
The spin analogue of this is the set of bar partitions $\lambda$ of $n$
such that $\lambda^{0}=\emptyset$, \emph{i.e.} without parts divisible
by $p$. The link between these two properties is enlightened by the
\emph{doubling} construction presented in~\cite{Morris-Yaseen}. Via this
process, a bar partition $\lambda$ of $n$ is associated to a partition
$D(\lambda)$ of $2n$ such that $\lambda^{0}=\emptyset$ if and only if
$D(\lambda)^{(p+1)/2}=\emptyset$.

The way the irreducible spin characters of $\widetilde G$ can be labeled
by $\overline p$-quotients of $w$, just like the irreducible characters
of $G$ are labeled by $p$-quotients of $w$ is another example of
Olsson's Principle.
It should therefore not come as a surprise that the perfect isometry of
Theorem~\ref{thm:isoparfaite} (which associates spin characters of
$\tSym_n$ and $\widetilde G$ corresponding to the same $\overline
p$-quotient of $w$), is precisely the spin analogue of the perfect
isometry of~\cite[Thm\,3.6]{BrGr} (which associates characters of
$\Sym_n$ and $G$ corresponding to the same $p$-quotient of $w$), 
including the sign $\delta_{\overline
p}(\lambda)(-1)^{|\lambda^0|}$ which is the spin analogue of
$\delta_p(\lambda)(-1)^{|\lambda^{(p+1)/2}|}$.

By the same token, the basic set of $\SI(\widetilde G)$ we construct
in Theorem~\ref{thm:basicset1} is the spin analogue of that of $G$ given
in~\cite[Thm\,4.3]{BrGr}. 

It should be noted that, while Olsson's Principle can be an inspiration
for results about $\tSym_n$, it is by no means a proof. However, it can
be a useful guide to indicate what the proof should be. This is indeed
the narrative behind the sequence given by~\cite{BrGr3}, \cite{livesey} and
the present paper to generalize to the spin case the proof
of~\cite{BrGr}, whose result's spin analogue was inspired to the authors
by Olsson's Principle several years ago.

\bigskip

{\bf Acknowledgements.} Part of this work was done at the CIRM in Luminy 
during a research in pairs stay. The authors wish to thank the CIRM gratefully
for their financial and logistical support. The first author is
supported by Agence Nationale de la Recherche Projet ACORT
ANR-12-JS01-0003.
The second author also acknowledges financial support from the
Engineering and Physical Sciences Research Council grant \emph{Combinatorial Representation Theory} EP/M019292/1. 

 \bibliographystyle{abbrv}
\bibliography{references}

\end{document}